\documentclass[11pt, a4paper]{amsart}

\usepackage{amsfonts,amsmath,amssymb, amscd,fullpage}
\usepackage{stmaryrd}
\usepackage[all]{xy}

\newtheorem{theorem}{Theorem}[section]
\newtheorem{lemma}[theorem]{Lemma}
\newtheorem{proposition}[theorem]{Proposition}
\newtheorem{corollary}[theorem]{Corollary}

\theoremstyle{definition}

\theoremstyle{remark}
\newtheorem{remark}[theorem]{Remark}

\newtheorem{example}[theorem]{Example}

\newcommand\pf{\begin{proof}}
\newcommand\epf{\end{proof}}

\newcommand\C{\mathbb{C}}

\newcommand\yd{\mathcal{YD}}

\newcommand\ext{\mathrm{Ext}}
\newcommand\co{\operatorname{co}}

\DeclareMathOperator{\id}{id}

\numberwithin{equation}{section}

\title{Cohomological dimensions of   universal  \\ cosovereign Hopf algebras}

\author{Julien Bichon}
\address{
Laboratoire de Math\'ematiques,
Universit\'e Blaise Pascal,
Campus universitaire des C\'ezeaux,
3 place Vasarely,
63178~Aubi\`ere Cedex, France}
\email{Julien.Bichon@math.univ-bpclermont.fr}


\begin{document}

\begin{abstract}
We compute the Hochschild and Gerstenhaber-Schack cohomological dimensions of the universal cosovereign Hopf algebras,  when the matrix of parameters is a generic asymmetry. Our main tools are considerations on the cohomologies of free product of Hopf algebras, and on the invariance of the cohomological dimensions under graded twisting by a finite abelian group.
\end{abstract}

\maketitle


\section{Introduction}

Given an algebra $A$, a well-known and important homological invariant of $A$ is its 
Hochschild cohomological dimension, which serves as a noncommutative analogue of the dimension of an algebraic variety, and is defined by
\begin{align*}
{\rm cd}(A)&= {\rm sup}\{n : \ H^n(A, M) \not=0 \ {\rm for} \ {\rm some} \ A-{\rm bimodule} \ M\}\in \mathbb N \cup \{\infty\} \\
& = {\rm min}\{n : \ H^{n+1}(A, M) =0 \ {\rm for} \ {\rm any} \ A-{\rm bimodule} \ M\}  \\
&={\rm pd}_{_A\mathcal M_A}(A) 
\end{align*}
where $H^*(A,-)$ denotes Hochschild cohomology and ${\rm pd}_{_A\mathcal M_A}(A)$ is the projective dimension of $A$ in the category of $A$-bimodules.

In this paper we will be interested in the case when $A$ is a Hopf algebra, in which case we have as well
$${\rm cd}(A)={\rm pd}_{A}(\mathbb C_\varepsilon)$$
where  ${\rm pd}_{A}(\mathbb C_\varepsilon)$ is the projective dimension of the trivial object $\mathbb C_\varepsilon$ in the category of (say right) $A$-modules (we work throughout the paper over the field of complex numbers). Hopf algebras simultaneously generalize, among other things, discrete groups and linear algebraic groups, and in the classical situations of of Hopf algebras associated to algebraic and discrete groups, the Hochschild cohomological dimensions are as follows.

\begin{enumerate}
\item  If $A=\mathcal O(G)$, the coordinate algebra on a linear algebraic group $G$, it is well-known that ${\rm cd}(\mathcal O(G))= \dim G$, the usual dimension of $G$.
\item If $A= \C\Gamma$, the group algebra of a discrete group $\Gamma$, then ${\rm cd}(\C\Gamma)={\rm cd}_\C(\Gamma)$, the cohomological dimension of $\Gamma$ with coefficients $\C$. This dimension of high importance in geometric group theory, see \cite{br,dw}. We have ${\rm cd}(\C\Gamma)=0$  if and only if $\Gamma$ is finite, and if $\Gamma$ is finitely generated, then ${\rm cd}(\C\Gamma)=1$ if and only if $\Gamma$ contains a free normal subgroup of finite index, see \cite{dun,di80,dw}.
\end{enumerate}

There is also another cohomology theory specific to Hopf algebras, Gerstenhaber-Schack cohomology \cite{gs1,gs2} (the coefficients are Hopf bimodules or Yetter-Drinfeld modules), which has been useful in proving some fundamental results in Hopf algebra theory \cite{ste2, eg}, and serves, similarly as above, to define another cohomological dimension, denoted ${\rm cd}_{\rm GS}(A)$. In \cite{bic,bi16} we proposed to study Gerstenhaber-Schack cohomology in order to get informations on Hochschild cohomology itself. For example it is proved in \cite{bi16} that 
${\rm cd}(A) \leq {\rm cd}_{\rm GS}(A)$ for any Hopf algebra, and it is asked there (Question 1.2) whether the equality always holds, at least in the cosemisimple case (a positive answer is provided in the cosemisimple Kac type case, i.e. when $S^2={\rm id}$). A positive answer would lead to the interesting fact that two Hopf algebras having equivalent monoidal  
categories of comodules have the same Hochschild cohomological dimension (the Gerstenhaber-Schack cohomological dimension being an invariant in such a situation).

The aim of this paper is to discuss the cohomological dimensions of a class of Hopf algebras that we believe to be of particular interest, the universal cosovereign Hopf algebras $H(F)$, $F \in {\rm GL}_n(\mathbb C)$, \cite{bi01} (see Section \ref{sec:ucha} for the precise definition). The universal property of these Hopf algebras make them play, in Hopf algebra theory, a role similar to that of the general linear groups in algebraic group theory, and to that of the free groups in discrete group theory. In particular any finitely generated cosemisimple Hopf algebra is a quotient of one $H(F)$.  

The Hopf algebras $H(F)$ are the algebraic counterparts of the universal compact quantum groups of Van Daele and Wang \cite{vdw}, which have been widely studied in the operator algebraic context of quantum group theory, starting with \cite{ba97}, see e.g. \cite{bra,dec,dfy,vaevan, vervoi}. However, so far, the algebraic properties of the general $H(F)$ have only been analyzed through the study of its category of comodules 
\cite{ba97, bi07, chi12, biri}. 
We provide here a full computation of the cohomological dimensions of $H(F)$, when the matrix $F$ is a generic asymmetry (see Section \ref{sec:ucha} for the definition of these notions): in that case we show that ${\rm cd}(H(F))=3={\rm cd}_{\rm GS}(H(F))$. To prove this result, our starting point is the recent observation \cite{bny} that for $F=E^tE^{-1}$, then $H(F)$ is a graded twisting of the free product $\mathcal B(E)*\mathcal B(E)$, where $\mathcal B(E)$ is the universal Hopf algebra of the bilinear form associated to $E$, introduced by Dubois-Violette and Launer \cite{dvl}. Since the cohomological dimensions of $\mathcal B(E)$ are known \cite{bic,bi16}, the computation is then achieved thanks to the two main general contributions of this paper: on one hand the invariance (for cosemisimple Hopf algebras) of the cohomological dimensions under graded twisting by a finite abelian group, and on the other hand the description of the cohomologies of a free product in terms of the cohomologies of the factors.

Note that for $F=I_n$, we have $H(I_n)=\mathcal O(U_n^+)$, the coordinate algebra on the free unitary quantum group $U_n^+$ \cite{wa95}. Thus, similarly to the case of the free orthogonal quantum group $O_n^+$ studied in \cite{cht, bic} and of the quantum permutation group $S_n^+$ studied in \cite{bi16}, we get that all the cohomological dimensions for $U_n^+$ equal $3$. Therefore the ``free'' quantum groups $O_n^+$, $S_n^+$, $U_n^+$ (see e.g. \cite{basp, fre} for the meaning of free) all have dimension $3$. It would be interesting to know is there is a conceptual reason (maybe representation theoretic, in the spirit of \cite{frwe}) for that.

The paper is organized as follows. In Section \ref{sec:ucha} we recall some basic facts on the universal cosovereign Hopf algebras, and we state our main result on the computation of their cohomological dimensions. Section \ref{sec:ydgs} provides the necessary material on Yetter-Drinfeld modules and Gerstenhaber-Schack cohomology.
In Section \ref{sec:gradim}, we recall, after some considerations on cocentral exact sequences of Hopf algebras and Yetter-Drinfeld modules over them, the construction of the graded twisting of a Hopf algebra, and prove the invariance of the cohomological dimensions under this construction (under suitable assumptions). Section \ref{sec:free} is devoted to the description of Hochschild and Gerstenhaber-Schack cohomologies of free products in terms of the cohomologies of the factors, and finishes the proof of Theorem \ref{thm:comput}. In the final Section \ref{sec:compare}, we come back to the problem of comparing the two cohomological dimensions for cosemisimple Hopf algebras, and provide a slight extension to the positive result in \cite{bi16}, proving that equality holds if $S^4={\rm id}$.

\medskip

\textbf{Notations and conventions.}
We work over $\mathbb C$ (or over any algebraically closed field of characteristic zero). 
We assume that the reader is familiar with the theory of Hopf algebras and their tensor categories of comodules, as e.g. in \cite{kas,ks,mon}.
If $A$ is a Hopf algebra, as usual, $\Delta$, $\varepsilon$ and $S$ stand respectively for the comultiplication, counit and antipode of $A$. We use Sweedler's notations in the standard way. The category of right $A$-comodules is denoted $\mathcal M^A$, the category of right $A$-modules is denoted $\mathcal M_A$, etc...  
The trivial (right) $A$-module is denoted $\C_\varepsilon$.
The set of $A$-module morphisms (resp. $A$-comodule morphisms) between two $A$-modules (resp. two $A$-comodules) $V$ and $W$ is denoted ${\rm Hom}_A(V,W)$ (resp. ${\rm Hom}^A(V,W)$).

\section{Universal cosovereign Hopf algebras}\label{sec:ucha}

We fix $n \geq 2$, and let $F \in {\rm GL}_n(\mathbb C)$.
Recall \cite{bi01} that the algebra $H(F)$ is the algebra
generated by
 $(u_{ij})_{1 \leq i,j \leq n}$ and
 $(v_{ij})_{1 \leq i,j \leq n}$, with relations:
$$ {u} {v^t} = { v^t} u = I_n ; \quad {vF} {u^t} F^{-1} = 
{F} {u^t} F^{-1}v = I_n,$$
where $u= (u_{ij})$, $v = (v_{ij})$ and $I_n$ is
the identity $n \times n$ matrix. The algebra
$H(F)$ has a  Hopf algebra structure
defined by
\begin{gather*}
\Delta(u_{ij}) = \sum_k u_{ik} \otimes u_{kj}, \quad
\Delta(v_{ij}) = \sum_k v_{ik} \otimes v_{kj}, \\
\varepsilon (u_{ij}) = \varepsilon (v_{ij}) = \delta_{ij}, \quad 
S(u) = {v^t}, \quad S(v) = F { u^t} F^{-1}.
\end{gather*}

The universal property of the Hopf algebras $H(F)$ \cite{bi01} shows that they
play, in the category of Hopf algebras,
a role that is similar to the one of $\mathcal O({\rm GL}_n(\C))$  
in the category of commutative Hopf algebras:  any 
finitely generated Hopf algebra having all its finite-dimensional comodules
isomorphic to their bidual (in particular any finitely generated cosemisimple Hopf algebra) is a quotient of $H(F)$ for some $F$.
Hence one might say that they correspond to ``universal'' quantum groups.

When  $F$ is a positive matrix, the Hopf algebra 
$H(F)$ is the canonical  Hopf $*$-algebra associated to  Van Daele and Wang's universal compact quantum groups \cite{vdw}. 

The category of comodules over $H(F)$ has been studied in \cite{ba97,bi07,chi12,biri}. 
In order to recall the characterization of  the cosemisimplicity of $H(F)$, we need some vocabulary.
A matrix $F \in {\rm GL}_n(\mathbb C)$ is said to be

$\bullet$ normalizable if ${\rm tr}(F) \not= 0$ and  $ {\rm tr} (F^{-1})\not=0$ or ${\rm tr}(F)=0={\rm tr} (F^{-1})$;

$\bullet$ generic if it is normalizable and the solutions of the equation
$q^2 -\sqrt{{\rm tr}(F){\rm tr}(F^{-1})}q +1 = 0$ are generic, i.e. are not roots of unity of order $\geq 3$ (this property does not depend on the choice of the square root);

$\bullet$ an asymmetry if there exists $E \in {\rm GL}_n(\mathbb C)$ such that $F=E^tE^{-1}$ (the terminology comes from the theory of bilinear forms, see \cite{rie}).

For $q \in \C^*$, we denote by $F_q$ the matrix 
$$F_q =  \begin{pmatrix} q^{-1} & 0 \\
                          0 & q \\
       \end{pmatrix}$$
and by $H(q)$ the Hopf algebra $H(F_q)$. The matrix $F_q$ is an asymmetry.
 
The following results are shown in \cite{bi07}.

$\bullet$ If $F$ is normalizable, we have an equivalence between the tensor categories of comodules
$$\mathcal M^{H(F)} \simeq^\otimes \mathcal M^{H(q)}$$
where $q$ is any solution of the equation $q^2 -\sqrt{{\rm tr}(F){\rm tr}(F^{-1})}q +1 = 0$.

$\bullet$ The Hopf algebra $H(F)$ is cosemisimple if and only if $F$ is generic.

Moreover, the simple comodules can be naturally labeled by the free monoid $\mathbb N * \mathbb N$ \cite{ba97, bi07, chi12}, with an explicit model for these comodules given in \cite{biri}.

The aim of this paper is to prove the following result.

\begin{theorem}\label{thm:comput}
 Let $F \in {\rm GL}_n(\C)$, $n\geq 2$.
\begin{enumerate}
 \item If $F$ is an asymmetry, then  ${\rm cd}(H(F))=3$.
\item If $F$ is generic, then ${\rm cd}_{\rm GS}(H(F))=3$.
\end{enumerate}
In particular, if $F$ is a generic asymmetry, we have ${\rm cd}(H(F))=3={\rm cd}_{\rm GS}(H(F))$.
\end{theorem}

See the next section for the definition of ${\rm cd}_{\rm GS}$. We will proceed by using the fact, from  \cite{bny}, that when $F$ is an asymmetry, it is a graded twisting of a free product of Hopf algebras whose cohomological dimension are known. We therefore have two main tasks:
\begin{enumerate}
 \item relate the cohomological dimensions of  (cosemisimple) Hopf algebras that are graded twisting of each other (this is done in Section \ref{sec:gradim});
\item describe the cohomological dimensions of a free product of Hopf algebras in terms of the cohomological dimensions of the factors (this is done in Section \ref{sec:free}).
\end{enumerate}




\section{Yetter-Drinfeld modules and Gerstenhaber-Schack cohomology}\label{sec:ydgs}

In this section we recollect the basic facts on Yetter-Drinfeld modules and Gerstenhaber-Schack cohomology, and discuss restriction and induction for Yetter-Drinfeld modules.
 Let $A$ be a Hopf algebra.

\subsection{Yetter-Drinfeld modules} Recall that a (right-right) Yetter-Drinfeld
module over $A$ is a right $A$-comodule and right $A$-module $V$
satisfying the condition, $\forall v \in V$, $\forall a \in A$, 
$$(v \cdot a)_{(0)} \otimes  (v \cdot a)_{(1)} =
v_{(0)} \cdot a_{(2)} \otimes S(a_{(1)}) v_{(1)} a_{(3)}$$
The category of Yetter-Drinfeld modules over $A$ is denoted $\yd_A^A$:
the morphisms are the $A$-linear $A$-colinear maps.
Endowed with the usual tensor product of 
modules and comodules, it is a tensor category, with unit the trivial Yetter-Drinfeld module, denoted $\C$.

We now discuss some important constructions of Yetter-Drinfeld modules (left-right versions of these constructions were first given in \cite{camizh97}, see \cite{shst} as well, in the context of Hopf bimodules).

Let $V$ be a right $A$-comodule. The Yetter-Drinfeld module $V \boxtimes A$ is defined as follows \cite{bic}. 
As a vector space $V \boxtimes A=V \otimes A$, the right module structure is given by multiplication on the right, and the right coaction $\alpha_{V \boxtimes A}$ is defined by
$$ \alpha_{V \boxtimes A}( v \otimes a)=v_{(0)}  \otimes a_{(2)} \otimes S(a_{(1)})v_{(1)}a_{(3)}$$
The coadjoint Yetter-Drinfeld module is $A_{\rm coad} = \mathbb C \boxtimes A$.

A Yetter-Drinfeld module is said to be free if it is isomorphic to $V\boxtimes A$ for some comodule $V$, and
 is said to be relative projective if it is a direct summand of a free Yetter-Drinfeld module.
If $A$ is cosemisimple, then the projective objects in the category $\yd_A^A$ are precisely the relative projective Yetter-Drinfeld modules, see \cite[Proposition 4.2]{bi16}, the abelian category $\yd_A^A$ has enough projectives, each object having a resolution by free Yetter-Drinfeld modules \cite[Corollary 3.4]{bic}.

\subsection{Gerstenhaber-Schack cohomology}
Let $V$ be a Yetter-Drinfeld module over $A$.
The Gerstenhaber-Schack cohomology of $A$ with coefficients in $V$, that we denote $H_{\rm GS}^*(A,V)$, was introduced in \cite{gs1,gs2} by using an explicit bicomplex. In fact Gerstenhaber-Schack used Hopf bimodules instead of Yetter-Drinfeld modules to define their cohomology, but in view of the equivalence between Hopf bimodules and Yetter-Drinfeld modules \cite{sc94}, we  work with the simpler framework of Yetter-Drinfeld modules.
A special instance of Gerstenhaber-Schack cohomology is bialgebra cohomology, given by 
$H_b^*(A)=H_{\rm GS}^*(A,\mathbb C)$.

As examples, the bialgebra cohomologies of $\C\Gamma$ (for a discrete group $\Gamma$) and 
of $\mathcal O(G)$ (for a connected reductive algebraic group $G$) are described in \cite{pw}. Some finite-dimensional examples are also computed in \cite{tai07}.

A key result, due to Taillefer \cite{tai04}, characterizes Gerstenhaber-Schack cohomology as  an $\ext$-functor:
$$H_{\rm GS}^*(A,V)\simeq \ext^*_{\yd_A^A}(\mathbb C, V)$$
We will use this description as a definition. Note that the category $\yd_A^A$ has enough injective objects \cite{camizh97, tai04}, so the above ${\rm Ext}$ spaces can be studied using injective resolutions of $V$. 

The Gerstenhaber-Schack cohomological dimension of a Hopf  algebra $A$ is
defined to be 
$${\rm cd}_{\rm GS}(A)= {\rm sup}\{n : \ H_{\rm GS}^n(A, V) \not=0 \ {\rm for} \ {\rm some} \ V \in \yd_A^A\}\in \mathbb N \cup \{\infty\}$$  
The following facts were established in \cite{bic,bi16}.

$\bullet$ ${\rm cd}(A) \leq {\rm cd}_{\rm GS}(A)$, with equality if $A$ is cosemisimple of Kac type (i.e. $S^2={\rm id}$).
 
$\bullet$ If $A$, $B$ are Hopf algebras with $\mathcal M^A \simeq^\otimes \mathcal M^B$ (the tensor categories of comodules are equivalent), then $\max({\rm cd}(A),{\rm cd}(B))\leq {\rm cd}_{\rm GS}(A)={\rm cd}_{\rm GS}(B)$.

$\bullet$ If $A$ is co-Frobenius (in particular if $A$ is cosemisimple), so that $\yd_A^A$ has enough projective objects, then ${\rm cd}_{\rm GS}(A)= {\rm pd}_{\yd_A^A}(\C)$.

\subsection{Restriction and induction for Yetter-Drinfeld modules}\label{subsec:res}
We now discuss the restriction and induction process for Yetter-Drinfeld modules, having in mind applications to Gerstenhaber-Schack cohomology. The considerations in this subsection (construction of a pair of adjoint functors) are special instances of those in  \cite[Section 2.5]{camizh}, but we give the detailed construction, on one hand for the sake of completeness, and on the other hand because it is probably quicker to write them down directly than to translate from the language of entwined modules of \cite{camizh}. 

Let $B \subset A$ be a Hopf subalgebra. For an $A$-comodule $X$, we put 
$$X^{(B)}=\{x \in X \ | \ x_{(0)} \otimes x_{(1)} \in X \otimes B\}$$
It is clear that the $A$-comodule structure on $X$ induces a $B$-comodule structure on $X^{(B)}$, and that this construction produces a functor
\begin{align*}
 \mathcal M^A & \longrightarrow \mathcal M^B \\
X & \longmapsto X^{(B)}
\end{align*}
This functor is left exact, and
we will say that $B \subset A$ is a coflat if this functor is exact (this agrees with the usual terminology, since the above functor is isomorphic with the functor $-\square_A B$). For example $B \subset A$ is coflat when $A$ is cosemisimple. 

\begin{proposition}\label{prop:funyd2}
 Let $B\subset A$ be a Hopf subalgebra, 
and  let $X$ be an object in $\yd_A^A$. Then $X^{(B)}$ is a sub-$B$-module of $X$, so that  $X^{(B)}$ is an object in $\yd_B^B$. The assignment 
\begin{align*}
 \yd^A_A & \longrightarrow \yd^B_B \\
X & \longmapsto X^{(B)}
\end{align*}
defines a linear functor, that we call the restriction functor, which is exact if $B \subset A$ is  coflat.
\end{proposition}

\begin{proof}
 For $x \in X^{(B)}$ and $b \in B$, we have
$$(x\cdot b)_{(0)} \otimes (x\cdot b)_{(1)}= x_{(0)} \cdot b_{(2)} \otimes S(b_{(1)})x_{(1)}b_{(3)} \in X \otimes B$$
and hence $X^{(B)}$ is a sub-$B$-module of $X$. The other assertions are immediate.
\end{proof}

We have an induction functor as well.

\begin{proposition}\label{prop:indyd}
 Let $B \subset A$ be a Hopf subalgebra. Then for any $V \in \yd_B^B$, the vector space $V\otimes_B A$ admits a natural Yetter-Drinfeld module structure over $A$, whose $A$-module structure is given by multiplication on the right, and whose $A$-comodule structure is given by the map
$$v \otimes_B a \mapsto v_{(0)} \otimes_B a_{(2)} \otimes S(a_{(1)}) v_{(1)} a_{(3)}$$
This construction defines a linear functor
\begin{align*}
 \yd_B^B &\longrightarrow \yd_A^A \\
V &\longmapsto V\otimes_B A
\end{align*}
that we call the induction functor.
\end{proposition}

\begin{proof}
 We have, for $v \in V$, $b \in B$ and $a \in A$,
\begin{align*}
(v\cdot b)_{(0)} \otimes_B a_{(2)} \otimes S(a_{(1)}) (v\cdot b)_{(1)} a_{(3)}& =
v_{(0)}\cdot b_{(2)} \otimes_B a_{(2)} \otimes S(a_{(1)})S(b_{(1)}) v_{(1)} b_{(3)}a_{(3)}\\
& = v_{(0)} \otimes_B b_{(2)}a_{(2)} \otimes S(b_{(1)}a_{(1)}) v_{(1)} b_{(3)}a_{(3)}
\end{align*}
and this shows that above map is well-defined. It is then straightforward to check that this indeed defines a comodule structure on $V\otimes_B A$, and a Yetter-Drinfeld module structure, and that we get the announced functor.
\end{proof}

We now observe that the functors of Propositions \ref{prop:funyd2} and \ref{prop:indyd} form a pair of adjoint functors, see \cite[Section 2.5]{camizh}.

\begin{proposition}\label{prop:adjyd}
 Let $B \subset A$ be a Hopf subalgebra. We have for any $V \in \yd_B^B$ and any $X \in \yd_A^A$, natural isomorphisms
$${\rm Hom}_{\yd_A^A}(V \otimes_B A, X) \simeq {\rm Hom}_{\yd_B^B}(V, X^{(B)})$$ 
If moreover $B \subset A$ is coflat and $A$ is flat as a left $B$-module, we have, for any $n\geq 0$, natural isomorphisms
$${\rm Ext}_{\yd_A^A}^n(V \otimes_B A, X) \simeq {\rm Ext}_{\yd_B^B}^n(V, X^{(B)})$$ 
\end{proposition}

\begin{proof}
It is a direct verification to check that for $f \in  {\rm Hom}_{\yd_A^A}(V \otimes_B A, X)$, the map
$f_0 : V \rightarrow X$ defined by $f_0(v)= f(v\otimes 1)$ has values into $X^{(B)}$, and is a morphism of Yetter-Drinfeld modules over $B$. We get a (natural) map
\begin{align*}
 {\rm Hom}_{\yd_A^A}(V \otimes_B A, X) & \longrightarrow {\rm Hom}_{\yd_B^B}(V, X^{(B)}) \\
f & \longmapsto f_0, \ f_0(v)= f(v\otimes 1)
\end{align*}
which is easily seen to be an isomorphism, and hence we have a pair of adjoint functors.

The assumptions that $B \subset A$ is coflat and that $A$ is flat as a left $B$-module are precisely that our pair of adjoint functors is formed by exact functors, and hence the restriction functor $\yd_A^A \rightarrow \yd_B^B$ preserve injective objects. Starting now from an injective resolution 
$$0 \rightarrow X \rightarrow I_0 \rightarrow I_1 \rightarrow \cdots$$
in $\yd_A^A$, we get an injective resolution 
$$0 \rightarrow X^{(B)} \rightarrow I_0^{(B)} \rightarrow I_1^{(B)} \rightarrow \cdots$$
in $\yd_B^B$, and the adjunction property gives an isomorphism of complexes
$${\rm Hom}_{\yd_A^A}(V \otimes_B A, I_*) \simeq {\rm Hom}_{\yd_B^B}(V, I_*^{(B)})$$ 
The ${\rm Ext}$-spaces in the statement are the cohomologies of these complexes.
\end{proof}

We finish the subsection by noticing that, in most cases, there is another description of the restriction functor (this will be convenient in the next section). 

\begin{proposition}\label{prop:rescoinv}
 Let $B \subset A$ be a Hopf subalgebra. Consider the quotient coalgebra $L = A/B^+A$, and denote $p: A \rightarrow L$ the quotient map. For $X \in \mathcal M^A$, put $$X^{{\rm co}L}= \{x \in X \ | \ x_{(0)}\otimes p(x_{(1)}) = x \otimes 1\}$$
If $B= A^{{\rm co} L}$, then we have $X^{{\rm co}L}=X^{(B)}$. Hence the assignment $X \mapsto X^{{\rm co}L}$ defines linear functors $\mathcal M^A \rightarrow \mathcal M^B$, $\yd_A^A \rightarrow \yd_B^B$, that are exact if $B \subset A$ is coflat, or if $L$ is cosemisimple.
\end{proposition}

\begin{proof}
  Given $x \in X^{{\rm co}L}$, we have $x_{(0)} \otimes x_{(1)} \otimes p(x_{(2)}) = x_{(0)} \otimes x_{(1)}\otimes 1$, and this shows that $x_{(0)} \otimes x_{(1)} \in X \otimes B$, since $B=A^{{\rm co} L}$. Conversely, if  $x_{(0)} \otimes x_{(1)} \in X \otimes B$, then  $x_{(0)} \otimes p(x_{(1)}) =  x_{(0)} \otimes \varepsilon(x_{(1)})=x \otimes 1$, hence $x \in X^{{\rm co}L}$. We get the announced description for $X^{{\rm co}L}$, the other assertions follow from Proposition \ref{prop:funyd2}, and there just remains to check exactness if $L$ is cosemisimple. For this, notice that if $X$ is a comodule over $A$, the coalgebra map $p$ induces a $L$-comodule structure on $X$, and the cosemisimplicity of $L$ provides a decomposition
$$X = X^{{\rm co}L} \oplus X'$$
for some sub-$L$-comodule $X'$.
 A morphism of $A$-comodules preserves this decomposition, and from this, exactness of our functor follows easily.
\end{proof}

\begin{remark}
The assumption  $B= A^{{\rm co} L}$ holds if $A$ is flat as a left $B$-module, see \cite[Corollary 1.8]{schn}.
\end{remark}

\section{Graded twisting and cohomological dimensions}\label{sec:gradim}

In this section we study  the invariance of the cohomological dimensions under graded twisting by a finite abelian group.

\subsection{Exact sequences of Hopf algebras} We begin by some preliminaries on exact sequences of Hopf algebras.
First recall that a sequence  of Hopf algebra maps
\begin{equation*}\C \to B \overset{i}\to A \overset{p}\to L \to
\C\end{equation*} is said to be exact \cite{ad} if the following
conditions hold:
\begin{enumerate}\item $i$ is injective and $p$ is surjective,
\item $\ker p =Ai(B)^+ =i(B)^+A$, where $i(B)^+=i(B)\cap{\rm Ker}(\varepsilon)$,
\item $i(B) = A^{\co L} = \{ a \in A:\, (\id \otimes p)\Delta(a) = a \otimes 1
\} = {^{\co L}A} = \{ a \in A:\, (p \otimes \id)\Delta(a) = 1 \otimes a
\}$. \end{enumerate}
Note that condition (2) implies  $pi= \varepsilon 1$.

In an exact sequence as above, we will assume, without loss of generality, that $B$ is Hopf subalgebra and $i$ is the inclusion map. In what follows we fix an exact sequence of Hopf algebras 
\begin{equation*}\C \to B \overset{i}\to A \overset{p}\to L \to
\C\end{equation*}
First we have the following well-known fact.

\begin{proposition}\label{prop:invexact}
 Let $M$ be a right $A$-module, and let $M^B=\{x \in M \ | \ x\cdot b=\varepsilon(b)x\}$ be the space of $B$-invariants. Then the $A$-module structure on $M$ induces an $L$-module structure on $M^B$ with $(M^B)^L = M^A$.  
\end{proposition}

\begin{proof}
 For $x \in M^B$ and $b \in B^+$, we have $x\cdot b=0$. Moreover, for $x \in M^B$, $a\in A$, one easily sees, using that $AB^+=B^+A$, that $x\cdot a \in M^B$.  
Hence the formula $x\cdot p(a)= x\cdot a$ provides a well-defined $L$-module structure on $M^B$. The last equality is immediate.
\end{proof}

\begin{proposition}\label{prop:moy}
  Assume  that $L$ is semisimple. Let $\tau \in L$ be a  right integral with $\varepsilon(\tau)=1$, and let $t \in A$ be such that $p(t)=\tau$. Let $V, W$ be right $A$-modules and let $f : V \rightarrow W$ be a $B$-linear map. Then the linear map $\tilde{f} : V \rightarrow W$ defined by $\tilde{f}(v)= f(v\cdot S(t_{(1)}))\cdot t_{(2)}$ is $A$-linear.
\end{proposition}

\begin{proof}
 Recall that ${\rm Hom}(V,W)$ admits a right $A$-module structure defined by 
$$f\cdot a(v)= f(v\cdot S(a_{(1)}))\cdot a_{(2)}$$
and that 
$${\rm Hom}_A(V,W)= {\rm Hom}(V,W)^A= ({\rm Hom}(V,W)^B)^L$$
Recall also that if $M$ is a right $L$-module over the semisimple algebra $L$, then $M^L=M\cdot \tau$.
Hence, since $f \in {\rm Hom}_B(V,W)={\rm Hom}(V,W)^B$, we have $f\cdot \tau \in ({\rm Hom}(V,W)^B)^L={\rm Hom}_A(V,W)$. We now have $f\cdot \tau = f \cdot p(t)=f\cdot t$, and it is clear that $f\cdot t$ is the map $\tilde{f}$ in the statement.
\end{proof}

\begin{remark}
When the above $p$ is an isomorphism, the above result is simply the well-known fact that a Hopf algebra having a right integral $\tau$ with $\varepsilon(\tau)=1$ is semisimple. 
\end{remark}

\subsection{Yetter-Drinfeld modules and cocentral exact sequences} 
Recall that a Hopf algebra map $p : A \rightarrow L$ is said to be cocentral if $p(a_{(1)}) \otimes a_{(2)} = p(a_{(2)}) \otimes a_{(1)}$ for any $a \in A$, and we say that an exact sequence $\C \to B \rightarrow A \overset{p}\rightarrow \C \Gamma \to \C$ is cocentral if $p$ is. 

In this subsection we fix a cocentral exact sequence of  Hopf algebras
$$\C \to B \rightarrow A \overset{p}\rightarrow \C \Gamma \to \C$$
with $\Gamma$ a finite abelian group. 
Our  aim is to relate Yetter-Drinfeld modules over $A$ and $B$, and then use these considerations to relate the cohomological dimensions of $A$ and $B$ (Theorem \ref{thm:cdexactcocentral}).

We assume that $A$ (and hence $B$) is cosemisimple (but this will play a true role only in Lemma \ref{lemm:integral} and Proposition \ref{prop:retract}).

Our first task is to study the action of the functor of Proposition \ref{prop:rescoinv} (and hence of Proposition \ref{prop:funyd2}) on relative projective Yetter-Drinfeld modules. We begin with the free ones.



\begin{lemma}\label{lemm:coinvfree}
 We have, for any $V \in \mathcal M^A$, 
$$(V \boxtimes A)^{{\rm co}\C \Gamma}= V^{{\rm co}\C \Gamma} \boxtimes A$$
\end{lemma}

\begin{proof}
 For $v \in V^{{\rm co}\C \Gamma}$ and $a \in A$, we have, using the cocentrality of $p$, 
\begin{align*}
 (v\otimes a)_{(0)} \otimes p((v \otimes a)_{(1)}) & = v_{(0)} \otimes a_{(2)} \otimes p(S(a_{(1)}) v_{(1)} a_{(3)}) \\
& = v \otimes a_{(2)} \otimes p(S(a_{(1)}) a_{(3)}) \\
& = v \otimes a \otimes 1
\end{align*}
Hence $V^{{\rm co}\C \Gamma} \boxtimes A \subset (V \boxtimes A)^{{\rm co}\C \Gamma}$. Conversely, let $\sum_iv_i \otimes a_i \in (V \boxtimes A)^{{\rm co}\C \Gamma}$. We have, using the cocentrality of $p$ and the fact that the algebra $\mathbb C\Gamma$ is commutative,
\begin{align*}
 \sum_iv_i \otimes a_i \otimes 1 &= \sum_i v_{i(0)} \otimes a_{i(2)} \otimes p(S(a_{i(1)}) v_{i(1)} a_{i(3)})\\ 
& = \sum_i v_{i(0)} \otimes a_{i} \otimes p( v_{i(1)})
\end{align*}
Taking the $a_i's$ linearly independent, we see that each $v_i$ belongs to $V^{{\rm co} \mathbb C \Gamma}$, as needed. 
\end{proof}

\begin{proposition}\label{prop:funydfree}
 The exact functor
\begin{align*}
(-)^{{\rm co}\C\Gamma} : \yd^A_A & \longrightarrow \yd^B_B \\
V & \longmapsto V^{{\rm co}\C\Gamma}
\end{align*}
transforms relative projective objects of $\yd^A_A$ into relative projective objects of $\yd_B^B$.
Moreover, if $M\in \yd_A^A$ is free, then the $B$-module structure on $M^{{\rm co}\C\Gamma}$ is the restriction of an $A$-module structure, so that $M^{{\rm co}\C\Gamma}$ is an object in $\yd_A^A$.
\end{proposition}

\begin{proof}
 Let $P$ be a relative projective Yetter-Drinfeld over $A$: there exists another Yetter-Drinfeld module $Q$ and an $A$-comodule $V$ such that
$P \oplus Q \simeq V \boxtimes A$ as Yetter-Drinfeld modules over $A$. We then have, using Lemma \ref{lemm:coinvfree},
$$P^{{\rm co}\C\Gamma} \oplus Q^{{\rm co}\C\Gamma} \simeq (P \oplus Q)^{{\rm co}\C\Gamma} \simeq (V \boxtimes A)^{{\rm co}\C\Gamma}\simeq V^{{\rm co}\C\Gamma}\boxtimes A$$
as Yetter-Drinfeld modules over $B$ (recall \cite[Proposition 4.5]{bi16} that if $W$ is $B$-comodule, then the free Yetter-Drinfeld module $W \boxtimes A\in \yd_A^A$ is in fact a comodule over $B$, so that $W \boxtimes A$ is Yetter-Drinfeld over $B$). Moreover $V^{{\rm co}\C\Gamma}\boxtimes A$ is a relative projective Yetter-Drinfeld module over $B$, by \cite[Proposition 4.8]{bi16}, hence there exists $T  \in \yd_B^B$ and $W \in \mathcal M^B$ such that 
$(V^{{\rm co}\C\Gamma}\boxtimes A) \oplus T \simeq W \boxtimes B$. Finally
$$P^{{\rm co}\C\Gamma} \oplus Q^{{\rm co}\C\Gamma} \oplus T \simeq W \boxtimes B$$
which shows that 
$P^{{\rm co}\C\Gamma}$ is indeed a relative projective Yetter-Drinfeld module over $B$. The last statement follows immediately from Lemma \ref{lemm:coinvfree}.
\end{proof}

Before proving our main technical result in view of the proof of Theorem \ref{thm:cdexactcocentral}, we need a last ingredient.

\begin{lemma}\label{lemm:integral}
 There exists an element $t \in A$ such that 
$$p(t)= \frac{1}{|\Gamma|}\sum_{g \in \Gamma} g \quad {\rm and} \quad t_{(2)}\otimes S(t_{(1)})t_{(3)}=t\otimes 1$$
If $M$, $N$ are $A$-modules and $f : M \rightarrow N$ is a $B$-linear map, then the linear map 
$\tilde{f} : M \rightarrow N$ defined by $\tilde f(x)=  r(x\cdot S(t_{(1)}))\cdot t_{(2)}$
is $A$-linear.
\end{lemma}

\begin{proof}
The element $\tau =  \frac{1}{|\Gamma|}\sum_{g \in \Gamma} g$ is a right integral in the semisimple Hopf algebra $\mathbb C\Gamma$ satisfying $\varepsilon(\tau)=1$, so the last statement follows from Proposition \ref{prop:moy}, and it remains to check that $t$ can be chosen such that  $ t_{(2)}\otimes S(t_{(1)})t_{(3)}=t\otimes 1$. 
To see this, note that $\mathbb C\Gamma$ and $A$ both admit right $B\otimes \mathbb C\Gamma$-comodule structures given by
\begin{align*}
 \mathbb C\Gamma &\longrightarrow \mathbb C\Gamma \otimes (B\otimes \mathbb C\Gamma), \quad \quad A \longrightarrow A \otimes (B\otimes \mathbb C\Gamma) \\
x & \longmapsto x_{(1)}\otimes 1 \otimes x_{(2)}, \quad \quad \ a \longmapsto a_{(2)} \otimes S(a_{(1)}) a_{(3)} \otimes p(a_{(4)})
\end{align*}
and that $p$ is $B\otimes \mathbb C\Gamma$-colinear. Since $A$ is cosemisimple, so is $B$ and  hence $B \otimes \mathbb C\Gamma$ is cosemisimple. Thus there exists $\sigma$, a $B\otimes \mathbb C\Gamma$-colinear section to $p$,  which thus satisfies, for any $x \in \mathbb C\Gamma$, 
$$\sigma(x)_{(2)} \otimes S(\sigma(x)_{(1)}) \sigma(x)_{(3)} \otimes p(\sigma(x)_{(4)})=
\sigma(x_{(1)}) \otimes 1 \otimes  x_{(2)}$$
and hence in particular
$$\sigma(x)_{(2)} \otimes S(\sigma(x)_{(1)}) \sigma(x)_{(3)} =
\sigma(x) \otimes 1 $$
Hence we can take $t = \sigma(\tau)$.
\end{proof}

\begin{proposition}\label{prop:retract}
 Let $V,W \in \yd_A^A$, and let $i :V \rightarrow W$ be an injective morphism of Yetter-Drinfeld modules over $A$.
Assume that the following conditions hold.
\begin{enumerate}
 \item We have $\alpha_V(V) \subset V \otimes B$ and $\alpha_W(W)\subset W \otimes B$, so that $V$ and $W$ are in fact $B$-comodules.
\item The exists a $B$-linear and $B$-colinear map $r : W \rightarrow V$ such that $ri={\rm id}_V$
\end{enumerate}
Then there exists an $A$-linear and $A$-colinear map $\tilde{r} : W \rightarrow V$ such that $\tilde{r}i={\rm id}_V$.
\end{proposition}

\begin{proof}
Let $t \in A$ as in the previous lemma.
We define  $\tilde{r} : W \rightarrow V$ by 
$$\tilde{r}(w) = r(w\cdot S(t_{(1)}))\cdot t_{(2)}$$
It is immediate to check that $\tilde{r}i={\rm id}_V$, and it follows from the previous lemma that $\tilde{r}$ is $A$-linear. It thus remains to check that $\tilde{r}$ is $A$-colinear. Let $w \in W$.
We have
\begin{align*}
 \tilde{r}(w)_{(0)} \otimes \tilde{r}(w)_{(1)}& = (r(w\cdot S(t_{(1)}))\cdot t_{(2)})_{(0)} \otimes (r(w\cdot S(t_{(1)}))\cdot t_{(2)})_{(1)} \\
&=r(w\cdot S(t_{(1)}))_{(0)}\cdot t_{(3)} \otimes S(t_{(2)})r(w\cdot S(t_{(1)}))_{(1)} t_{(4)}\\
&=r((w\cdot S(t_{(1)}))_{(0)})\cdot t_{(3)} \otimes S(t_{(2)})(w\cdot S(t_{(1)}))_{(1)} t_{(4)}\\
& = r(w_{(0)}\cdot S(t_{(2)}))\cdot t_{(5)} \otimes S(t_{(4)})S^{2}(t_{(3)})w_{(1)} S(t_{(1)}) t_{(6)}\\
& = r(w_{(0)}\cdot S(t_{(2)}))\cdot t_{(3)} \otimes w_{(1)}S(t_{(1)})t_{(4)} \\
& = r(w_{(0)}\cdot S(t_{(1)}))\cdot t_{(2)} \otimes w_{(1)} \\
& =  \tilde{r}(w_{(0)}) \otimes w_{(1)}
\end{align*}
where we have used the Yetter-Drinfeld condition, the $B$-colinearity of $r$ and the fact that $t_{(2)}\otimes S(t_{(1)})t_{(3)}=t\otimes 1$ which gives 
$$t_{(2)}\otimes t_{(3)} \otimes S(t_{(1)})t_{(4)}=t_{(1))} \otimes t_{(2)}\otimes 1$$
 Hence $\tilde{r}$ is $A$-colinear, and this concludes the proof.
\end{proof}

\begin{theorem}\label{thm:cdexactcocentral}
Let  $\C \to B \rightarrow A \rightarrow \C \Gamma \to \C$ be a cocentral exact sequence of Hopf algebras, with $\Gamma$ a finite abelian group. Then ${\rm cd}(A)={\rm cd}(B)$, and if we assume that $A$ is cosemisimple, we have 
 ${\rm cd}_{\rm GS}(A)={\rm cd}_{\rm GS}(B)$ as well.
\end{theorem}

\begin{proof}
 The identity ${\rm cd}(A)={\rm cd}(B)$ is a particular case of \cite[Proposition 3.2]{bi16} (the above  exact sequence being automatically strict, see \cite{bi16} for details). Assume that $A$ is cosemisimple, and consider a resolution of the trivial Yetter-Drinfeld module
$$\cdots  \rightarrow P_n \rightarrow P_{n-1} \rightarrow \cdots\rightarrow  P_1 \rightarrow P_0 \rightarrow \mathbb C$$
by (relative) projective Yetter-Drinfeld over $A$. The exact functor $(-)^{{\rm co}\C\Gamma} : \yd^A_A  \longrightarrow \yd^B_B$ from Proposition \ref{prop:rescoinv}
transforms, by Proposition \ref{prop:funydfree}, this resolution into a resolution of $\mathbb C$ by (relative) projective Yetter-Drinfeld modules over $B$. It follows that  ${\rm cd}_{\rm GS}(B)\leq{\rm cd}_{\rm GS}(A)$.

To prove the converse inequality, we can assume that $m={\rm cd}_{\rm GS}(B)$ is finite.
 Consider a resolution of the trivial Yetter-Drinfeld module
$$\cdots  \rightarrow F_n \rightarrow F_{n-1} \rightarrow \cdots\rightarrow  F_1 \rightarrow F_0 \rightarrow \mathbb C$$
by free Yetter-Drinfeld modules over $A$. The exact functor $(-)^{{\rm co}\C\Gamma} : \yd^A_A  \longrightarrow \yd^B_B$ from Proposition \ref{prop:rescoinv}
transforms, by Lemma \ref{lemm:coinvfree} and 
Proposition \ref{prop:funydfree}, this resolution into a resolution of $\mathbb C$
of type
$$\cdots  \rightarrow V_n \boxtimes A \rightarrow V_{n-1}\boxtimes A \rightarrow \cdots\rightarrow  V_1 \boxtimes A\rightarrow V_0 \boxtimes A\rightarrow \mathbb C$$
where $V_0, V_1, \ldots$ are comodules over $B$, $V_0\boxtimes A, V_1\boxtimes A, \ldots$ are free Yetter-Drinfeld modules over $A$ (and projective over $B$), and the involved linear map are morphisms of Yetter-Drinfeld modules over $A$. Since $m={\rm cd}_{\rm GS}(B)$,  a standard argument yields an exact sequence of Yetter-Drinfeld modules over $B$, and hence over $A$
$$0 \rightarrow K \overset{i}\rightarrow V_m \boxtimes A \rightarrow V_{m-1}\boxtimes A \rightarrow \cdots\rightarrow  V_1 \boxtimes A\rightarrow V_0 \boxtimes A\rightarrow \mathbb C$$
together with $r : V_m \boxtimes A \rightarrow K$, a morphism of Yetter-Drinfeld modules over $B$ such that $ri = {\rm id}_K$. Proposition \ref{prop:retract} gives $\tilde{r} : V_m \boxtimes A \rightarrow K$, a morphism of Yetter-Drinfeld  
modules over $A$ such that $\tilde{r}i={\rm id}_K$. We thus get, since a direct summand of a projective is projective, a length $m$ resolution of $\mathbb C$ by projective Yetter-Drinfeld modules over $A$, and we conclude that ${\rm cd}_{\rm GS}(A)\leq m$, as required.
\end{proof}

\subsection{Graded twisting} 
Let $A$ be a Hopf algebra and let $\Gamma$ be a  group. Recall \cite{bny} that an invariant cocentral action of $\Gamma$ on $A$ is a pair $(p, \alpha)$ where
\begin{itemize}
\item $p\colon A \rightarrow \mathbb C\Gamma$ is a surjective cocentral Hopf algebra map,
\item $\alpha\colon \Gamma \rightarrow {\rm Aut}_{\rm Hopf}(A)$ is an action of $\Gamma$ by Hopf algebra automorphisms on $A$, with $p\alpha_g=p$ for all $g \in \Gamma$.
\end{itemize}
The Hopf algebra map $p$ induces a $\Gamma$-grading on $A$
$$A = \bigoplus_{g \in \Gamma} A_g, \quad A_g=\{ a \in A \ | \ p(a_{(1)}) \otimes a_{(2)}= g \otimes 1 \}$$ 
and the last condition is equivalent to $\alpha_g(A_h)= A_h$ for all $g,h \in \Gamma$. When $(p,\alpha)$ is such an action, the graded twisting $A^{t,\alpha}$ of $A$ is the Hopf subalgebra
$$
A^{t, \alpha}=\sum_{g \in \Gamma} A_g \otimes g\subset A \rtimes \Gamma,
$$
of the crossed product Hopf algebra $A \rtimes \Gamma$. Notice that the coalgebras $A$ and $A^{t, \alpha}$ are isomorphic.

With these definitions, we are ready to prove the following result.

\begin{theorem}\label{thm:cdgrad}
 Let $A$, $B$ be Hopf algebras, and assume that $B$ is a graded twisting of $A$ by a finite abelian group. Then 
${\rm cd}(A)={\rm cd}(B)$, and if we assume that $A$ is cosemisimple, we have 
 ${\rm cd}_{\rm GS}(A)={\rm cd}_{\rm GS}(B)$ as well.
\end{theorem}

\begin{proof}
Let $\Gamma$ be the twisting group. It pointed out in \cite{bny2} that $A$ and  $B$ fit into  cocentral exact sequences
$$\C \to L \rightarrow A \rightarrow \C \Gamma \to \C, \quad \C \to L \rightarrow B \rightarrow \C \Gamma \to \C$$
for the Hopf algebra $L=A_1$. We thus have, by Theorem \ref{thm:cdexactcocentral}, ${\rm cd}(A)={\rm cd}(L)={\rm cd}(B)$ and ${\rm cd}_{\rm GS}(A)={\rm cd}_{\rm GS}(L)={\rm cd}_{\rm GS}(B)$ if $A$ (and hence $B$) is cosemisimple. 
\end{proof}

See \cite{bny} for examples of graded twistings. The example we have in mind is the following one.

Let $E\in {\rm GL}_n(\mathbb C)$ and consider the Hopf algebra $\mathcal B(E)$ defined by Dubois-Violette and Launer \cite{dvl}: $\mathcal B(E)$ is the algebra generated by $a_{ij}$, $1 \leq i,j \leq n$, subject to the relations
$E^{-1}a^tEa=I_n=aE^{-1}a^tE$, where $a$ is the matrix $(a_{ij})$ (for an appropriate matrix $E_q$, one gets  $\mathcal O_q({\rm SL}_2(\mathbb C))$).

 Recall now \cite{wa95} that if $A, B$ are Hopf algebras, the free product algebra $A*B$ has a unique Hopf algebra structure such that the canonical morphisms $A \rightarrow A*B$ and $B\rightarrow A*B$ are Hopf algebra maps, and consider the free product Hopf algebra $\mathcal B(E) * \mathcal B(E)$.
We have a cocentral Hopf algebra map
\begin{align*}
 \mathcal B(E) * \mathcal B(E) & \longrightarrow \mathbb C \mathbb Z_2, \ \mathbb Z_2 =\langle g \rangle \\
a_{ij}^{(1)}, a_{ij}^{(2)} & \longmapsto \delta_{ij} g
\end{align*}
where the superscript refers to the numbering of copies inside the free product, and we have an action of $\mathbb Z_2$ of $\mathcal B(E) * \mathcal B(E)$, given by the Hopf algebra automorphism that exchanges the copies inside the free product. We get in this way an invariant cocentral action of $\mathbb Z_2$ on $\mathcal B(E) * \mathcal B(E)$, and we form the graded twisting $(\mathcal B(E) * \mathcal B(E))^t$.
Now for $F=E^{t}E^{-1}$, we have, by \cite{bny}, a Hopf algebra isomorphism
\begin{align*}
H(F) &\longrightarrow (\mathcal B(E) * \mathcal B(E))^t \\
u, \ v  &\longmapsto a^{(1)} \otimes g, \  E^t a^{(2)} (E^{-1})^t \otimes g
\end{align*}
Using Theorem \ref{thm:cdexactcocentral}, 
we get:

\begin{corollary}\label{cor:cdtw}
For $F=E^{t}E^{-1}$, we have ${\rm cd}(H(F))={\rm cd}(\mathcal B(E) * \mathcal B(E))$, and if $F$ is generic, then 
${\rm cd}_{\rm GS}(H(F))={\rm cd}_{\rm GS}(\mathcal B(E) * \mathcal B(E))$.
\end{corollary}
 
 It thus remains, in order to prove Theorem \ref{thm:comput}, to discuss the cohomological dimensions of  a free product, the cohomological dimensions of $\mathcal B(E)$ being known \cite{bic}.

\section{Cohomologies of free products of Hopf algebras}\label{sec:free}

In this section we discuss the cohomologies of a free product of Hopf algebras.



\subsection{Hochschild cohomology}
We begin with Hochschild cohomology.  The following result generalizes the well-known one for group cohomology, see \cite{hs}, with essentially the same proof. 

\begin{theorem}\label{thm:freeext}
 Let $A$, $B$ be augmented algebras. We have for any right $A*B$-module $M$, and for any $n \geq 2$, a natural isomorphism
$${\rm Ext }_{A*B}^n(\mathbb C_\varepsilon, M) \simeq {\rm Ext }_{A}^n(\mathbb C_\varepsilon, M) \oplus {\rm Ext }_{B}^n(\mathbb C_\varepsilon, M)$$ where $M$ has the respective restricted $A$-module and $B$-module structures.
\end{theorem}

The result and proof are probably well-known, but in lack of a convenient reference (however see \cite{tow} for the case of trivial coefficients), we will give the details, which also will be useful in view of the proof of an analogous result for Gerstenhaber-Schack cohomology in a forthcoming  subsection.

At the level of Hochschild cohomology, Theorem \ref{thm:freeext} gives the following result.

\begin{theorem}\label{thm:freehoc}
 Let $A$, $B$ be Hopf algebras algebras. We have for any  $A*B$-bimodule $M$, and for any $n \geq 2$, a natural isomorphism
$$H^n(A*B, M) \simeq H^n(A, M) \oplus H^n(B, M)$$ where $M$ has the respective restricted $A$-bimodule and $B$-bimodule structures.
\end{theorem}

\begin{proof}
 This follows directly from Theorem \ref{thm:freeext}, since for a Hopf algebra $A$ and an $A$-bimodule $M$, we have $H^*(A,M) \simeq {\rm Ext}_A^*(\mathbb C_\varepsilon, M')$, where $M'$ has the right $A$-module structure given by $m\cdot a= S(a_{(1)})\cdot m \cdot a_{(2)}$. See e.g. \cite{bic} for this well-known fact.
\end{proof}

\begin{corollary}\label{cor:freecd}
 Let $A$, $B$ be non trivial Hopf algebras. We have
$${\rm cd}(A*B)= \begin{cases}
                  1 & \text{if ${\rm cd}(A)=0={\rm cd}(B)$} \\
{\rm max}({\rm cd}(A), {\rm cd}(B)) & \text{if ${\rm max}({\rm cd}(A), {\rm cd}(B)) \geq 1$} 
                 \end{cases}$$
\end{corollary}

\begin{proof}
 If ${\rm cd}(A)=0={\rm cd}(B)$, Theorem \ref{thm:freeext} yields that ${\rm cd}(A*B)\leq 1$. If ${\rm cd}(A*B)=0$, then $A*B$ is a semisimple Hopf algebra, hence is finite-dimensional. But if $A$ and $B$ are non trivial, the free product algebra $A*B$ is necessarily infinite-dimensional, so ${\rm cd}(A*B)=1$.

Since $A*B$ is free both as a left $A$-module and as a left $B$-module, we have ${\rm cd}(A*B)\geq {\rm max}({\rm cd}(A), {\rm cd}(B))$ (see e.g. \cite{bi16}). If $m={\rm max}({\rm cd}(A), {\rm cd}(B)) \geq 1$, we get 
from Theorem \ref{thm:freeext} that $H^{m+1}(A*B,M)=(0)$ for any $A$-bimodule $M$, and hence ${\rm cd}(A*B)\leq m$, as needed.
\end{proof}

\begin{remark}
 The above result also can be obtained as a direct consequence of \cite[Corollary 2.5]{ber}, as well. Indeed, the Hochschild cohomological dimension of a Hopf algebra coincides with its (right or left) global dimension (this is pointed out in \cite{wyz}), hence the result follows from Corollary 2.5 in \cite{ber}, having in mind that since  $A$ and $B$ are non trivial, the free product algebra $A*B$ is necessarily infinite-dimensional, so is not semisimple, and hence ${\rm cd}(A*B)\geq 1$. 

We prefer the proof obtained as a corollary of Theorem \ref{thm:freehoc}, since it gives more information on Hochschild cohomology, and it will be better for adaptation to the Gerstenhaber-Schack cohomology case. 
\end{remark}

Our 
aim now is to prove Theorem \ref{thm:freeext}. We
begin with the following classical shifting lemma.

\begin{lemma}\label{lemm:shift}
 Let $(A,\varepsilon)$ be an augmented algebra. We have, for any $n\geq 2$ and any right $A$-module $M$,
$${\rm Ext}^n_A(\mathbb C_\varepsilon, M) \simeq {\rm Ext}^{n-1}_A( A^+, M)$$
\end{lemma}

\begin{proof}
 One gets the result by applying the ${\rm Ext}$ long exact sequence to the exact sequence of right $A$-modules
$$ 0 \to A^+ \to A \overset{\varepsilon}\to \mathbb C_{\varepsilon} \to 0$$  
with $A$ a free $A$-module, hence projective.
\end{proof}

We thus have to study the augmentation ideal in a free product of augmented algebras.
For this, recall that if $(A, \varepsilon)$ is an augmented algebra and  $M$ is a right $A$-module, a derivation $d : A \rightarrow M$ is a linear map  such that $d(ab)=\varepsilon(a)d(b)+d(a)b$ for any $a,b \in A$. The space  of such derivations is denoted ${\rm Der}(A,M)$.

The following result relates derivations and the augmentation ideal.

\begin{lemma}\label{lemm:deraug}
 Let $A=(A, \varepsilon)$ be an augmented algebra. We have, for any  right $A$-module $M$, a natural isomorphism
$${\rm Der}(A,M)\simeq {\rm Hom}_{A}(A^+,M)$$
\end{lemma}

\begin{proof}
 It follows from the definition of a derivation that we have a linear map
\begin{align*}
 {\rm Der}(A,M)& \longrightarrow {\rm Hom}_{A}(A^+,M) \\
d &\longmapsto d_{|A^+}
\end{align*}
which is easily seen to be an isomorphism.
\end{proof}

\begin{lemma}\label{lemm:derfree}
 We have, for any augmented algebras $A$, $B$ and for any right $(A*B)$-module $M$, natural isomorphisms
$${\rm Der}(A*B,M)\simeq {\rm Der}(A,M)\oplus {\rm Der}(B,M)$$
\end{lemma}

\begin{proof}
We have a  linear map
\begin{align*}
 {\rm Der}(A*B,M)&\longrightarrow {\rm Der}(A,M)\oplus {\rm Der}(B,M) \\
d & \longmapsto (d_{|A}, d_{|B})
\end{align*}
which is clearly injective by the derivation property and the fact that $A*B$ is generated, as an algebra, by $A$ and $B$. To prove surjectivity, first recall that in general, if $M$ is a right $A$-module, a derivation $d : A \rightarrow M$ corresponds precisely to an algebra map
$$
 A \longrightarrow \begin{pmatrix} A & M \\ 0 & A\end{pmatrix}, \ 
a \longmapsto \begin{pmatrix} a & d(a) \\0 & a\end{pmatrix}$$
where $\begin{pmatrix} A & M \\ 0 & A\end{pmatrix}$ is the usual triangular matrix algebra, with $M$ having the left $A$-module structure induced by $\varepsilon$. Now given $(d,d') \in {\rm Der}(A,M)\oplus {\rm Der}(B,M)$, consider the algebra maps
$$
 A \longrightarrow \begin{pmatrix} A*B & M \\ 0 & A*B\end{pmatrix}, \ 
a \longmapsto \begin{pmatrix} a & d(a) \\0 & a\end{pmatrix}, \quad B \longrightarrow \begin{pmatrix} A*B & M \\ 0 & A*B\end{pmatrix}, \ 
b \longmapsto \begin{pmatrix} b & d'(b) \\0 & b\end{pmatrix}
$$
The universal property of the free product yields an algebra map
$$A*B \longrightarrow \begin{pmatrix} A*B & M \\ 0 & A*B\end{pmatrix}$$
which extends the above maps, and hence a derivation $\delta : A*B \rightarrow M$, which clearly satisfies $\delta_{|A}=d$ and $\delta_{|B}= d'$.
\end{proof}

\begin{lemma}\label{lemm:augfree}
Let $A,B$ be augmented algebras. We have
$$(A*B)^+ \simeq \left(A^+\otimes_A (A*B)\right) \oplus \left(B^+\otimes_B (A*B)\right)$$
as right $(A*B)$-modules
\end{lemma}

\begin{proof}
 We have, for any right $(A*B)$-module $M$, using lemmas \ref{lemm:deraug} and \ref{lemm:derfree},
\begin{align*}
 {\rm Hom}_{A*B}((A*B)^+,M) & \simeq {\rm Der}(A*B,M) \\
&  \simeq {\rm Der}(A,M)\oplus {\rm Der}(B,M) \\
& \simeq {\rm Hom}_{A}(A^+,M) \oplus {\rm Hom}_{B}(B^+,M) \\
& \simeq {\rm Hom}_{A*B}\left(A^+\otimes_A (A*B),M\right) \oplus {\rm Hom}_{A*B}\left(B^+\otimes_B (A*B),M\right) \\
& \simeq {\rm Hom}_{A*B}\left((A^+\otimes_A (A*B))\oplus (B^+\otimes_B (A*B)),M\right) 
\end{align*}
We conclude by the Yoneda Lemma.
\end{proof}

\begin{proof}[Proof of Theorem \ref{thm:freeext}]
 Let $n\geq 2$ and let $M$ be a right $A*B$-module. We have, using lemmas \ref{lemm:shift} and \ref{lemm:augfree},
\begin{align*}
{\rm Ext }_{A*B}^n(\mathbb C_\varepsilon, M) & \simeq {\rm Ext}^{n-1}_{A*B}( (A*B)^+, M) \\
& \simeq {\rm Ext}^{n-1}_{A*B}\left((A^+\otimes_A A*B) \oplus (B^+\otimes_B A*B), M\right) \\
& \simeq  {\rm Ext}^{n-1}_{A*B}(A^+\otimes_A (A*B), M) \oplus {\rm Ext}^{n-1}_{A*B}(B^+\otimes_B (A*B), M) \\
\end{align*}
Now since $A*B$ is flat as an $A$-module (it is even free as an $A$-module), \cite[Proposition 12.2, IV]{hs}
gives
$${\rm Ext}^{n-1}_{A*B}(A^+\otimes_A (A*B), M) \simeq {\rm Ext}^{n-1}_{A}(A^+, M)$$
and similarly for $B$. We thus have, using again Lemma \ref{lemm:shift},
\begin{align*}
{\rm Ext }_{A*B}^n(\mathbb C_\varepsilon, M) 
 & \simeq  {\rm Ext}^{n-1}_{A}(A^+, M) \oplus {\rm Ext}^{n-1}_B(B^+, M) \\
& \simeq {\rm Ext}^{n}_{A}(\mathbb C_\varepsilon, M) \oplus {\rm Ext}^{n}_B(\mathbb C_{\varepsilon}, M)
\end{align*}
which is the expected result.
\end{proof}

\subsection{Gerstenhaber-Schack cohomology}
We can now formulate the Gerstenhaber-Schack cohomology analogue of Theorem 5.2, using the restriction functor from Subsection \ref{subsec:res}.

\begin{theorem}\label{thm:freegs}
 Let $A$, $B$ be cosemisimple Hopf algebras algebras. We have, for any  Yetter-Drinfeld module $M$ over  $A*B$, and for any $n \geq 2$, a natural isomorphism
$$H^n_{\rm GS}(A*B, M) \simeq H^n_{\rm GS}(A, M^{(A)}) \oplus H^n_{\rm GS}(B, M^{(B)})$$ 
\end{theorem}

\begin{corollary}\label{cor:freecdgs}
 Let $A$, $B$ be non trivial cosemisimple Hopf algebras. We have
$${\rm cd}_{\rm GS}(A*B)= \begin{cases}
                  1 & \text{if ${\rm cd}_{\rm GS}(A)=0={\rm cd}_{\rm GS}(B)$} \\
{\rm max}({\rm cd}_{\rm GS}(A), {\rm cd}_{\rm GS}(B)) & \text{if ${\rm max}({\rm cd}_{\rm GS}(A), {\rm cd}_{\rm GS}(B)) \geq 1$} 
                 \end{cases}$$
\end{corollary}

\begin{proof}
 This is similar to the proof of Corollary \ref{cor:freecd}.
\end{proof}

The scheme of the   proof of Theorem \ref{thm:freegs} will be similar to that of Theorem \ref{thm:freeext}. First we have a shifting lemma. 

\begin{lemma}\label{lemm:shiftgs}
 Let $A$ be a cosemisimple Hopf algebra algebra. We have, for any $n\geq 2$ and any Yetter-Drinfeld module $M$ over  $A$,
$${\rm Ext}^n_{\yd_A^A}(\mathbb C, M) \simeq {\rm Ext}^{n-1}_{\yd_A^A}( A^+, M)$$
\end{lemma}

\begin{proof}
Similarly to Lemma \ref{lemm:shift}, the result is obtained by applying the ${\rm Ext}$ long exact sequence to the exact sequence of Yetter-Drinfeld modules
$$ 0 \to A^+ \to A_{\rm coad} \overset{\varepsilon}\to \mathbb C \to 0$$  
with $A_{\rm coad}$ a free Yetter-Drinfeld module, hence projective since $A$ is cosemisimple.
\end{proof}

Given a Yetter-Drinfeld module $M$ over $A$, we note by ${\rm Der}_{\yd}(A,M)$ the derivations $d\in {\rm Der}(A,M)$ such that $d : A_{\rm coad} \rightarrow M$ is a morphism of $A$-comodules. With this notation, we have the following analogue of Lemma \ref{lemm:deraug}, whose proof is immediate.

\begin{lemma}\label{lemm:deraugyd}
 Let $A$ be a Hopf algebra. We have, for any  $M \in \yd_A^A$, a natural isomorphism
$${\rm Der}_{\yd}(A,M)\simeq {\rm Hom}_{\yd_A^A}(A^+,M)$$
\end{lemma}

\begin{lemma}\label{lemm:derfreeyd}
 We have, for any Hopf algebras $A$, $B$ and for any  $M \in \yd_{A*B}^{A*B}$, natural isomorphisms
$${\rm Der}_{\yd}(A*B,M)\simeq {\rm Der}_{\yd}(A,M^{(A)})\oplus {\rm Der}_{\yd}(B,M^{(B)})$$
\end{lemma}

\begin{proof}
 Given $d \in {\rm Der}_{\yd}(A*B,M)$, it is immediate that $d_{|A} \in  {\rm Der}_{\yd}(A,M^{(A)})$, hence we have an injective linear map
\begin{align*}
 {\rm Der}_{\yd}(A*B,M)&\longrightarrow {\rm Der}_{\yd}(A,M^{(A)})\oplus {\rm Der}_{\yd}(B,M^{(B)}) \\
d & \longmapsto (d_{|A}, d_{|B})
\end{align*}
The proof of Lemma \ref{lemm:derfree} provides, for $(d,d') \in {\rm Der}_\yd(A,M^{(A)})\oplus {\rm Der}_\yd(B,M^{(B)})$, an element   $\delta \in {\rm Der}(A*B,M)$ such that $(\delta_{|A},\delta_{|B})=(d,d')$, with
$\delta \in {\rm Der}_{\yd}(A*B,M)$, so our map is surjective, and the proof is complete. 
\end{proof}

We now  describe the augmentation ideal of a free product, as a Yetter-Drinfeld module, using the induction functor from Subsection \ref{subsec:res}.




\begin{lemma}\label{lemm:augfreeyd}
Let $A,B$ be Hopf algebras. We have
$$(A*B)^+ \simeq \left(A^+\otimes_A (A*B)\right) \oplus \left(B^+\otimes_B (A*B)\right)$$
as Yetter-Drinfeld modules over $A*B$.
\end{lemma}

\begin{proof}
 We have, for any  $M \in \yd_{A*B}^{A*B}$, using lemmas \ref{lemm:deraugyd}, \ref{lemm:derfreeyd}, and Proposition \ref{prop:adjyd}, 
\begin{align*}
 {\rm Hom}_{\yd_{A*B}^{A*B}}((A*B)^+,M) & \simeq {\rm Der}_{\yd}(A*B,M) \\
&  \simeq {\rm Der}_{\yd}(A,M^{(A)})\oplus {\rm Der}_{\yd}(B,M^{(B)}) \\
& \simeq {\rm Hom}_{\yd_A^A}(A^+,M^{(A)}) \oplus {\rm Hom}_{\yd_B^B}(B^+,M^{(B)}) \\
& \simeq {\rm Hom}_{\yd_{A*B}^{A*B}}\left(A^+\otimes_A (A*B),M\right) \oplus {\rm Hom}_{\yd_{A*B}^{A*B}}\left(B^+\otimes_B (A*B),M\right) \\
& \simeq {\rm Hom}_{\yd_{A*B}^{A*B}}\left((A^+\otimes_A (A*B))\oplus (B^+\otimes_B (A*B)),M\right) 
\end{align*}
We conclude by the Yoneda Lemma.
\end{proof}

\begin{proof}[Proof of Theorem \ref{thm:freegs}]
 Let $n\geq 2$ and let $M$ be a Yetter-Drinfeld module over $A*B$. We have, using Lemma \ref{lemm:shiftgs}, Lemma \ref{lemm:augfreeyd} and Proposition \ref{prop:adjyd} (using that $A,B\subset A*B$ is flat and coflat),
\begin{align*}
{\rm Ext }_{\yd_{A*B}^{A*B}}^n(\mathbb C, M) & \simeq {\rm Ext}^{n-1}_{\yd_{A*B}^{A*B}}( (A*B)^+, M) \\
& \simeq {\rm Ext}^{n-1}_{\yd_{A*B}^{A*B}}\left((A^+\otimes_A A*B) \oplus (B^+\otimes_B A*B), M\right) \\
& \simeq  {\rm Ext}^{n-1}_{\yd_{A*B}^{A*B}}(A^+\otimes_A (A*B), M) \oplus {\rm Ext}^{n-1}_{\yd_{A*B}^{A*B}}(B^+\otimes_B (A*B), M)\\
& \simeq  {\rm Ext}^{n-1}_{\yd_A^A}(A^+, M^{(A)}) \oplus {\rm Ext}^{n-1}_{\yd_B^B}(B^+, M^{(B)}) \\
& \simeq {\rm Ext}^{n}_{\yd_A^A}(\mathbb C, M^{(A)}) \oplus {\rm Ext}^{n}_{\yd_B^B}(\mathbb C, M^{(B)})
\end{align*}
The result then follows from the ${\rm Ext}$-description of Gerstenhaber-Schack cohomology.
\end{proof}






\subsection{Application to $H(F)$: proof of Theorem \ref{thm:comput}} Let $F \in {\rm GL}_n(\mathbb C)$be an asymmetry ($n \geq 2$), so that $F=E^tE^{-1}$ for some $E \in {\rm GL}_n(\mathbb C)$.
By Corollary \ref{cor:cdtw}, we have ${\rm cd}(H(F))={\rm cd}(\mathcal B(E) * \mathcal B(E))$. We have ${\rm cd}(\mathcal B(E))=3$ by \cite{bic}, hence we get ${\rm cd}(H(F))=3$ from Corollary \ref{cor:freecd}.

Assume now that $F$ is generic, so that $H(F)$ and $\mathcal B(E)$ are cosemisimple. Then Corollary \ref{cor:cdtw} yields ${\rm cd}_{\rm GS}(H(F))={\rm cd}_{\rm GS}(\mathcal B(E) * \mathcal B(E))$. 
We have ${\rm cd}_{\rm GS}(\mathcal B(E))=3$, by \cite{bic,bi16}, hence by Corollary \ref{cor:freecdgs}, we obtain ${\rm cd}_{\rm GS}(H(F))=3$.

Assume finally that $F$ is generic, but not necessarily an asymmetry. Then there exists a generic asymmetry
$F'$ such that the tensor categories of comodules over $H(F)$ and $H(F')$ are equivalent (see Section 2), hence 
the monoidal invariance of ${\rm cd}_{\rm GS}$ (see \cite{bi16}) and the previous discussion ensure
that ${\rm cd}_{\rm GS}(H(F))={\rm cd}_{\rm GS}(H(F'))=3$, as required

\begin{remark}
 Suppose again that $F$ is generic, but not an asymmetry. Since  ${\rm cd}_{\rm GS}(H(F)) \geq  {\rm cd}(H(F))$ (\cite{bi16}), we get  ${\rm cd}(H(F)) \leq 3$. We conjecture that this is an equality.
\end{remark}

\begin{remark}
Consider the case $F=I_n$, so that $H(I_n) =\mathcal O(U_n^+)$, the coordinate algebra on the free unitary quantum group $U_n^+$. It follows immediately from Theorem \ref{thm:comput} that $\beta_k^{(2)}(\widehat{U_n^+})=0$ for $k \geq 4$, where $\beta_k^{(2)}$ stands for the $k$-th $L^2$-Betti number \cite{ky08} of the dual discrete quantum group $\widehat{U_n^+}$. It was shown in \cite{ver} that $\beta_1^{(2)}(\widehat{U_n^+})\not=0$, and this result has been made more precise in the recent preprint \cite{kyra}, where it is shown that $\beta_1^{(2)}(\widehat{U_n^+})=1$.
\end{remark}



\section{Relations between cohomological dimensions}\label{sec:compare}

In this last section we come back to the problem of comparing the two cohomological dimensions. We prove the following slight generalization of \cite[Corollary 5.10]{bi16}.

\begin{theorem}\label{thm:comparecd}
If $A$ is a cosemisimple Hopf algebra with $S^4={\rm id}$, we have ${\rm cd}(A)={\rm cd}_{\rm GS}(A)$.
\end{theorem}

Of course, a generalization of this theorem to the arbitrary cosemisimple case would make trivial the proof of the second part of Theorem \ref{thm:comput}.

Before proving this result, we need to recall some facts on the structure of cosemisimple Hopf algebras, see \cite[Chapter 11]{ks} for example. Let $A$ be a cosemisimple Hopf algebra with Haar integral $h : A \rightarrow \mathbb C$. There exists a convolution invertible linear map $\psi : A \rightarrow \mathbb C$ such that 
\begin{enumerate}
 \item $S^2= \psi * {\rm id}_A * \psi^{-1}$,
\item $\psi \circ S=\psi^{-1}$,
\item $\sigma = \psi * {\rm id}_{A} * \psi$ is an algebra endomorphism of $A$, and $h(ab)=h(b\sigma(a))$ for any $a,b \in A$.
\end{enumerate}

In all the known examples,  the map $\psi$ above can be chosen to be an algebra map, so that the second condition is automatic, but it is always
 unknown whether this can always be done (this is a particular case of Question 4.8.3 in \cite{egno}).
We call such a map $\psi$ a modular functional on $A$. 

\begin{lemma}\label{lemm:modu}
 Let $A$ be cosemisimple Hopf algebra with Haar integral $h$ and modular functional $\psi$.
We have for any $a,x \in A$
$$h(S(a_{(1)})xa_{(2)}) = \psi^{-1}(a_{(2)})\psi^{-1}(a_{(3)})h\left(xa_{(4)}S^{-1}(a_{(1)})\right)
$$ 
$$h(S(a_{(2)})xS^2(a_{(1)})) = \psi^{-1}(a_{(2)})\psi^{-1}(a_{(3)})h\left(a_{(1)}S(a_{(4)})x\right)$$
In particular, if $S^4={\rm id}$, we have $\psi*\psi= \varepsilon$ and
$$h(S(a_{(1)})xa_{(2)}) 
=h(x)\varepsilon(a)= h(S(a_{(2)})xS^2(a_{(2)}))$$
\end{lemma}

\begin{proof}
 We have, using that $S=\psi*S^{-1}*\psi^{-1}$,
 \begin{align*}
  h(S(a_{(1)})x(a_{(2)})) & = h(xa_{(2)}\sigma S(a_{(1)})) \\
& = h\left(xa_{(4)}\psi^{-1}(a_{(3)})S(a_{(2)})\psi^{-1}(a_{(1)})\right) \\
&= h\left(xa_{(6)}\psi^{-1}(a_{(5)})\psi(a_{(2)})S^{-1}(a_{(3)})\psi^{-1}(a_{(4)})\psi^{-1}(a_{(1)})\right) \\
& =  h\left(xa_{(4)}\psi^{-1}(a_{(3)})\psi^{-1}(a_{(2)})S^{-1}(a_{(1)})\right)\\
&= \psi^{-1}(a_{(2)})\psi^{-1}(a_{(3)})h\left(xa_{(4)}S^{-1}(a_{(1)})\right)
 \end{align*}
The second identity is obtained from the first one using that $hS=h$. At $x=1$, we get in particular
$$\varepsilon(a)=\psi^{-1}(a_{(2)})\psi^{-1}(a_{(3)})h\left(a_{(4)}S^{-1}(a_{(1)})\right)$$ 
If $S^{4}={\rm id}$, then $\psi * \psi$ convolution commutes with ${\rm id}_A$, hence we indeed see from the previous identity that that $\psi*\psi=\varepsilon$, and the last identities follow directly.
\end{proof}



\begin{remark}
The last condition in the lemma does not hold in general. 
For example it does not hold for $\mathcal O_q({\rm SL}_2(\mathbb C))$ if $q \not =\pm 1$.
\end{remark}

\begin{proposition}
 Let $V,W$ be Yetter-Drinfeld modules over the cosemisimple Hopf algebra $A$ satisfying $S^4={\rm id}$, let $i : W \rightarrow V$ be an injective morphism of Yetter-Drinfeld modules, and let $r : V \rightarrow W$ be an $A$-linear map such that 
$r i ={\rm id}_W$. Then there exists a morphism of Yetter-Drinfeld modules $\tilde{r} : V \rightarrow W$ such that $\tilde{r}i={\rm id}_W$. 
\end{proposition}

\begin{proof}
  Let $h$ be the Haar integral on $A$. Recall that for any $A$-comodules $V$ and $W$, we have a surjective averaging operator 
\begin{align*}
 M : {\rm Hom}(V,W) & \longrightarrow {\rm Hom}^A(V,W) \\
f & \longmapsto M(f), \ M(f)(v)= h\left(f(v_{(0)})_{(1)}S(v_{(1)})\right)f(v_{(0)})_{(0)}
\end{align*}
with $f \in {\rm Hom}^A(V,W)$ if and only if $M(f)=f$. We put $\tilde{r} = M(r)$, and it is straightforward to check that $\tilde{r}i={\rm id}_W$. It remains to check that $\tilde r$ is $A$-linear. We have, using the Yetter-Drinfeld condition and the $A$-linearity of $r$,
\begin{align*}
 \tilde{r}(v \cdot a) & = h\left(r((v\cdot a)_{(0)})_{(1)}S((v\cdot a)_{(1)})\right)r((v\cdot a)_{(0)})_{(0)}\\
& = h\left(r(v_{(0)}\cdot a_{(2)})_{(1)}S(S(a_{(1)})v_{(1)}a_{(3)})\right)r(v_{(0)} \cdot a_{(2)})_{(0)}\\
& = h\left((r(v_{(0)})\cdot a_{(2)})_{(1)}S(S(a_{(1)})v_{(1)}a_{(3)})\right)(r(v_{(0)}) \cdot a_{(2)})_{(0)}\\
& = h\left(S(a_{(2)})r(v_{(0)})_{(1)}a_{(4)}S(S(a_{(1)})v_{(1)}a_{(5)})\right)r(v_{(0)})_{(0)} \cdot a_{(3)}\\
& = h\left(S(a_{(2)})r(v_{(0)})_{(1)}a_{(4)}S(a_{(5)})S(v_{(1)})S^2(a_{(1)})\right)r(v_{(0)})_{(0)} \cdot a_{(3)}\\
& = h\left(S(a_{(2)})r(v_{(0)})_{(1)}S(v_{(1)})S^2(a_{(1)})\right)r(v_{(0)})_{(0)} \cdot a_{(3)}
\end{align*}
Thus, if $S^4={\rm id}$, Lemma \ref{lemm:modu} ensures that
\begin{align*}
 \tilde{r}(v \cdot a) & = h\left(S(a_{(2)})r(v_{(0)})_{(1)}S(v_{(1)})S^2(a_{(1)})\right)r(v_{(0)})_{(0)} \cdot a_{(3)} \\
& = h\left(r(v_{(0)})_{(1)}S(v_{(1)})\right)r(v_{(0)})_{(0)} \cdot a \\
&= \tilde{r}(v) \cdot a
\end{align*}
and hence $\tilde{r}$ is $A$-linear.
\end{proof}

\begin{proof}[Proof of Theorem \ref{thm:comparecd}]
 We already know that  ${\rm cd}(A)\leq{\rm cd}_{\rm GS}(A)$, and to prove the equality we can assume that $m={\rm cd}(A)$ is finite. Consider a resolution of the trivial Yetter-Drinfeld module
$$\cdots  \rightarrow P_n \rightarrow P_{n-1} \rightarrow \cdots\rightarrow  P_1 \rightarrow P_0 \rightarrow \mathbb C$$
by projective Yetter-Drinfeld modules over $A$. These are in particular projective as $A$-modules, so since $m={\rm cd}(A)$, a standard argument, once again, yields an exact sequence of Yetter-Drinfeld modules over $A$
$$0 \rightarrow K \overset{i}\rightarrow P_m \rightarrow P_{m-1} \rightarrow \cdots\rightarrow  P_1  \rightarrow P_0 \rightarrow \mathbb C$$
together with $r : P_m  \rightarrow K$, an $A$-linear map such that $ri = {\rm id}_K$. The previous proposition yields a morphism of Yetter-Drinfeld module $\tilde{r} : P_m \rightarrow K$ such that $\tilde{r}i={\rm id}_K$.
 We thus obtain, since a direct summand of a projective is projective, a length $m$ resolution of $\mathbb C$ by projective Yetter-Drinfeld modules over $A$, and we conclude that ${\rm cd}_{\rm GS}(A)\leq m$, as required.
\end{proof}

We get the following generalization of \cite[Corollary 5.11]{bi16}, with the same proof.

\begin{corollary}\label{cor:invcdS4}
 Let  $A$ and $B$ be cosemisimple Hopf algebras  such that there exists an equivalence of linear tensor categories
 $\mathcal M^A \simeq^{\otimes} \mathcal M^B$. If the antipode of $A$ satisfies $S^4={\rm id}$, then we have ${\rm cd}(A) \geq {\rm cd}(B)$, and if the antipodes of $A$ and $B$ both satisfy $S^4={\rm id}$, then ${\rm cd}(A) = {\rm cd}(B)$.
\end{corollary}

\begin{example}
As an application  of Theorem \ref{thm:comparecd}, consider, for $m,n \geq 1$, the  $(m+n) \times (m+n)$ matrix
$$I_{m,n}= \begin{pmatrix}
            I_m & 0 \\
0 & -I_n
           \end{pmatrix}$$
We have $S^4={\rm id}$ for the Hopf algebra $H(I_{m,n})$, since $I_{m,n}^2=I_{m+n}$, and $S^2\not={\rm id}$. For $q$ satisfying $q^2-(m-n)q +1=0$, we have $\mathcal M^{H(I_{m,n})} \simeq^\otimes \mathcal M^{H(q)}$. Hence $H(I_{m,n})$ is cosemisimple if $|m-n|\geq 2$, and in this case we have
$${\rm cd}(H(I_{m,n}))= {\rm cd}_{\rm GS}(H(I_{m,n}))={\rm cd}_{\rm GS}(H(q))=3$$
while $I_{m,n}$ is not an asymmetry if $n$ is odd.
\end{example}

To conclude, it is interesting to note that the question of a generalization of Corollary \ref{cor:invcdS4} (positive answer to Question 1.1 in \cite{bi16}) is studied as well in the recent preprint \cite{wyz}, in the setting of Hopf algebras having an homological duality, with interesting partial positive answers.



\end{document}